\documentclass[a4paper, 11pt]{amsart}
\usepackage[utf8]{inputenc}
\usepackage[T1]{fontenc}
\usepackage{lmodern}
\usepackage[english]{babel}
\usepackage[dvipsnames]{xcolor}
\usepackage{adjustbox}
\usepackage{letltxmacro}
\usepackage{todonotes}
\LetLtxMacro\todonotestodo\todo
\renewcommand{\todo}[2][]{\todonotestodo[#1]{TODO: {#2}}}
\usepackage{enumerate}
\usepackage{verbatim}

\usepackage[pdftex]{hyperref}
\hypersetup{
    colorlinks=true,
    linkcolor={blue},
    citecolor={magenta},
    urlcolor={blue}
}
\usepackage{breakurl}
\usepackage{url}
\usepackage{mathrsfs}
\usepackage{amssymb}
\usepackage{ stmaryrd }

\theoremstyle{definition}

\newtheorem{lemma}{Lemma}[section]
\newtheorem{theorem}[lemma]{Theorem}
\newtheorem{proposition}[lemma]{Proposition}
\newtheorem{corollary}[lemma]{Corollary}
\newtheorem{fact}[lemma]{Fact}

\newtheorem{remark}[lemma]{Remark}

\newtheorem{example}[lemma]{Example}

\newtheorem{definition}[lemma]{Definition}
\newtheorem*{claim*}{Claim}

\newtheorem*{theorem*}{Theorem}
\newtheorem*{corollary*}{Corollary}
\newtheorem*{lemma*}{Lemma}
\newtheorem*{remark*}{Remark}
\newtheorem*{question*}{Question}
\newtheorem{question}[lemma]{Question}

\newtheorem{problem}[lemma]{Problem}
\newtheorem{expl}[lemma]{Explanation}

\newcommand{\Z}{\mathbb{Z}}

\newcommand{\R}{\mathbb{R}}
\newcommand{\C}{\mathbb{C}}

\DeclareMathOperator{\Spec}{Spec}
\DeclareMathOperator{\Max}{Max}

\newcommand{\zar}[1]{\overline{#1}^{\text{Zar}}}

\title{Dependency equilibria: Boundary cases and their real algebraic geometry}

\author{Irem Portakal}
\address{Max Planck Institute for Mathematics in the Sciences Leipzig}
\email{mail@irem-portakal.de}

\author{Daniel Windisch}
\address{School of Mathematics, University of Edinburgh}
\email{daniel.windisch@ed.ac.uk}

\begin{document}

\maketitle

\begin{abstract}
This paper is a significant step forward in understanding dependency equilibria within the framework of real algebraic geometry encompassing both pure and mixed equilibria. In alignment with Spohn's original definition of dependency equilibria, we propose two alternative definitions, allowing for an algebro-geometric comprehensive study of all dependency equilibria. We give a sufficient condition for the existence of a pure dependency equilibrium and show that every Nash equilibrium lies on the Spohn variety, the algebraic model for dependency equilibria. For generic games, the set of real points of the Spohn variety is Zariski dense. Furthermore, every Nash equilibrium in this case is a dependency equilibrium.  Finally, we present a detailed analysis of the geometric structure of dependency equilibria for $(2\times2)$-games.
\end{abstract}

\section{Introduction}

A well-known concept in game theory, the Nash equilibrium, describes a state in which no player can improve their payoff by unilaterally changing their strategy. However, in some games, Nash equilibria do not lead to the most optimal outcomes, as in the case of the Prisoner’s Dilemma. Spohn addresses this issue by introducing the notion of dependency equilibria (DE) in 2003 to rationalize cooperation in the Prisoner’s Dilemma. This concept suggests that players' actions influence each other in a causal loop. The notion of DE can be understood through reflexive decision theory, where players appear to share a common cause in the standard non-reflexive model \cite{Spohn2003, spohn2024}. Spohn studied DE \cite{Spohn2007}, highlighting its non-trivial and fundamentally different nature compared to Nash equilibria, as even in $(2\times 2)$-games, it is no longer described by linear polynomials \cite{kidambi2025elliptic}. This started the fundamental geometric study of DE by the first author and Sturmfels in \cite{Irem-Bernd}. In this paper, the first connection between algebraic statistics and game theory was established. The classical notion of graphical models was used to examine the so-called conditional independence equilibria that model different possible dependencies between the players. This gave rise to universality theorems for conditional independence equilibria \cite{portakal2024game,portakal2024nash}, similar to the universality for Nash equilibria \cite{datta2003universality}. In the following, we present three possible formal objections against the general approach carried out in these publications and how we address them throughout the paper to make the results applicable for game theoretic purposes.
\begin{enumerate}
\item The authors considered all complex points satisfying the equations for the dependency equilibria of a given game, that is, the complex projective Spohn variety of the game (\cite[Theorem 3.2]{Irem-Bernd}, Definition~\ref{def: Spohn variety}). However, the main interest of game theorists, epistemologists, and people in economics concerns, of course, the set of real non-negative points because these are the ones that potentially can be DE at all. 
\begin{enumerate}
    \item[$\bullet$] In Theorem~\ref{theorem:realdense}, we prove that for generic games, the set of real points of the Spohn variety is Zariski dense. Consequently, the results on the complex Spohn variety from \cite[Theorem 3.2]{Irem-Bernd} for generic games also hold for its set of real points.
\end{enumerate}
\item Spohn's definition of DE, see~\cite{Spohn2003,Spohn2007}, involves conditional probabilities and, therefore, divisions by terms which are potentially equal to $0$. He fixes the definition by using limits. Since the paper~\cite{Irem-Bernd} was intended as a first step towards the geometric study of DE, this technicality was left unregarded and points where such a ``division by $0$'' happens were not considered. Thus, only the totally mixed equilibria were studied. 
\begin{enumerate}
    \item[$\bullet$] In Definition~\ref{def:de}, we give two definitions for general DE and we prove that every dependency equilibrium of a game lies on the Spohn variety (Lemma~\ref{lemma:DEinV(2x2)}). Theorem~\ref{Theorem:tangent} gives a criterion, in terms of linear algebra, for the existence of a whole family of totally mixed DE around a pure strategy. This criterion also serves as a sufficient condition for the existence of a pure DE.
    \item[$\bullet$] In Theorem~\ref{theorem:NEonSpohn}, we prove that every Nash equilibrium of a game $X$ lies on the Spohn variety of $X$. Moreover, if no irreducible component of the Spohn variety is contained in the hyperplanes generated by the denominators appearing in the definition of DE (conditional probabilities that are potentially zero), then every Nash equilibrium is a dependency equilibrium (Corollary~\ref{corollary:NEisDEgeneric}).
\end{enumerate}    
\item The payoff tables are most of the time assumed to be generic. In real-life situations, many common models (e.g.\ Example~\ref{example:one-generic-one-reducible}(b), Example~\ref{ex: Bach or Stravinski}) are not generic.  
\begin{enumerate}
\item[$\bullet$] In addition to the previous two points, we show that the Spohn variety can include points on the boundary of the probability simplex that are not DE. Moreover, there may be more DE than just the points on the Spohn variety whose conditional payoffs are well-defined (Theorem~\ref{theorem:bounds}). However, generically, these distinctions do not need to be considered (Proposition~\ref{proposition:constructible}). Throughout the paper, we also present detailed computations of various non-generic examples of games. In Section~\ref{section:(2x2)}, we specifically examine arbitrary $(2\times 2)$-games.
\end{enumerate}
\end{enumerate}

\section{Preliminaries and first examples}\label{section:explanation}

The purpose of this section is to introduce some game-theoretic notation for the following sections and its interpretation in terms of algebraic varieties. These geometric objects and their generalizations (such as schemes) are well-studied in mathematics and, more specifically, in algebraic geometry, and their investigation has evolved to a far-reaching theory. For references to the topic, see, for instance, \cite{classical-AG} introducing classical algebraic geometry and \cite{Wedhorn} developing the theory of schemes.
The algebraic variety that we are concerned with in order to study dependency equilibria (DE) is called the Spohn variety. We set up some notation before its definition.

The ambient space of the Spohn variety is a complex projective space $\mathbb{P}^s$ of dimension $s = d_1 \cdots d_n-1$, where $n$ is the number of players and $d_i$ is the number of pure strategies of player $i$. The reader not so familiar with projective spaces might just think of the ambient space as $\C^s$. 
Each player has a payoff table which is a tensor $X^{(i)}$ of format $d_1 \times \cdots \times d_n$ with real entries. For $2$-player games, these are two $(d_1 \times d_2)$-matrices. Here, $X^{(i)}_{j_1 \ldots j_n}$ represents the payoff that player $i$ receives when player $1$ chooses strategy $j_1$, player $2$ chooses strategy $j_2$, etc. The tuple $X = (X^{(i)})$ is typically called a \textit{game in normal form} and it can be interpreted as a point in $\R^{n\times d_1 \times \ldots \times d_n}$.

We consider (\textit{joint}) \textit{strategies} $p$, that is, joint probability distributions over the set of pure strategies of the players: The entry $p_{j_1\ldots j_n}$ is the probability that the first player chooses the strategy $j_1$, the second player $j_2$, etc. We will frequently interpret $p$ as a point in $\mathbb P^s$ that admits a representation by non-negative real coordinates. We denote by $\overline{\Delta}$ the set of all joint strategies and by $\Delta$ the set of all \textit{totally mixed strategies}, that is, strategies with all entries $p_{j_1\ldots j_n} >0$. The finitely many strategies with one entry equal to $1$ (and hence all others equal to $0$) are called \textit{pure}.  A DE is a strategy $p$ where each player $i \in [n] = \{1,\ldots,n\}$ maximizes their conditional expected payoff $\mathbb{E}^{(i)}_{k}(p)$ conditioned on them having fixed pure strategy $k \in [d_i]$. In the following, the hat means that the $i$th sum is removed:

\[
\displaystyle \mathbb{E}^{(i)}_{k}(p):=\sum_{j_1 = 1}^{d_1} \ldots \widehat{\sum_{j_{i} = 1}^{d_{i}}} \ldots \sum_{j_n = 1}^{d_n} X^{(i)}_{j_1\cdots k \cdots  j_n}\frac{p_{j_1\cdots k \cdots j_n}}{p_{+\cdots+k+\cdots+}},
\]
where $$p_{+\cdots+k+\cdots+}:=
\displaystyle\sum_{j_1 = 1}^{d_1} \ldots \widehat{\sum_{j_{i} = 1}^{d_{i}}} \ldots \sum_{j_n = 1}^{d_n}
p_{j_1\cdots k \cdots j_n}.
$$

\begin{definition}
   Let $X$ be a $(d_1 \times \cdots \times d_n)$-game in normal form. A joint probability distribution $p$ is called a \textit{dependency equilibrium} (DE) if $\mathbb{E}^{(i)}_{k}(p) \geq \mathbb{E}^{(i)}_{k'}(p)$ for all $k, k' \in [d_i]$ and for all $i \in [n]$.
\end{definition}

The above definition uses the condition $\mathbb{E}^{(i)}_{k}(p) \geq \mathbb{E}^{(i)}_{k'}(p)$ for all $k, k' \in [d_i]$. Hence, similarly to~\cite[Equation (4)]{Irem-Bernd}, we could also replace the inequality by an equality.

One may describe the totally mixed DE of a game $X$ as the points of an algebraic variety associated to $X$, the Spohn variety $\mathcal{V}_X$ of $X$, that have positive real coordinates \cite[Theorem 3.2]{Irem-Bernd}. The restriction to totally mixed DE was needed as the denominators of the conditional expected payoffs might vanish. We will treat the general case of not necessarily totally mixed strategies in Section~\ref{section:arbitraryGames} and we will see that the Spohn variety also plays a prominent role there.

\begin{definition}\label{def: Spohn variety}
The \textit{Spohn variety} $\mathcal{V}_X \subseteq \mathbb P^s$ of the game $X$ is the complex projective variety defined by the $(2\times 2)$-minors of the following $(d_i \times 2)$-matrices  $M_1, \ldots, M_n$ of linear forms:
 \begin{equation*}\label{eq: Spohn matrices}
M_i \,=\, M_i(p) \,\,:= \,\,\,\begin{bmatrix}
\vdots & \vdots \\
\,\,p_{+\cdots + k + \cdots+}\, &\,\, \displaystyle \sum_{j_1=1}^{d_1}\cdots \sum_{j_{i-1}=1}^{d_{i-1}} \sum_{j_{i+1}=1}^{d_{i+1}}\cdots \sum_{j_n=1}^{d_n} X^{(i)}_{j_1 \cdots  k  \cdots j_n} p_{j_1 \cdots  k  \cdots j_n}\, \\
\vdots & \vdots 
\end{bmatrix} \! .
\end{equation*}
\end{definition}

Throughout the paper, we will use the term ``generic'' in the sense of algebraic geometry. More precisely, a property or a statement holds \textit{generically} if it holds on an open (and hence dense) subset of the ambient space with respect to the Zariski topology. In our case, the ambient space will always be the real affine space $\R^{n\times d_1 \times \ldots \times d_n}$ of all games of a fixed size. In this setting, a statement like ``property $P$ holds for generic games'' translates as ``the set of games $X \in \R^{n\times d_1 \times \ldots \times d_n}$ satisfying property $P$ is Zariski open''. Note that the usual notion of a generic game from game theory is similar to our notion of genericity, where \( P \) is defined as ``the set of Nash equilibria of \( X \) consists of isolated points''. However, these two concepts of genericity may differ and should be studied separately. One example is ``Bach or Stravinsky'' (Example~\ref{ex: Bach or Stravinski}): Whereas this game has isolated points as its set of Nash equilibria, it is not generic in the sense of DE, since its Spohn variety is a reducible curve.

In \cite[Theorem 3.2]{Irem-Bernd}, it is shown that, for generic games, the Spohn variety is irreducible of codimension $d_1+\ldots+d_n - n$ and degree $d_1 \cdots d_n$ in $\mathbb P^s$. In any case, the dimension is at least the first number and the degree is at most the second number. This follows from Bézout's Theorem, a standard result in algebraic geometry. Understanding the Spohn variety is the first step in understanding how the set of DE of a game might look like.

\begin{example}\label{example:one-generic-one-reducible}
Let us consider $(2 \times 2)$-games in normal form, that is, games with two players each of which has two pure strategies, also called bimatrix games. For simplicity, we denote the players' $(2 \times 2)$-payoff matrices as \[
\begin{pmatrix}
a_{11} & a_{12}\\
a_{21} & a_{22} 
\end{pmatrix} \text{ and }
\begin{pmatrix}
b_{11} & b_{12}\\
b_{21} & b_{22} 
\end{pmatrix}.
\]
The Spohn variety $\mathcal{V}_X$ is defined by the determinants of $M_1$ and $M_2$ 
\[M_1 = \begin{bmatrix}
p_{11} + p_{12} & a_{11} p_{11} +a_{12} p_{12} \\
p_{21} + p_{22}  & a_{21} p_{21} +a_{22} p_{22} 
\end{bmatrix}, \ M_2 = \begin{bmatrix}
p_{11} + p_{21} & b_{11} p_{11}  +b_{21} p_{21} \\
 p_{1 2} + p_{22} & b_{12} p_{12} +b_{22} p_{22} \\

\end{bmatrix}\]

\noindent which give the two following defining equations
\begin{align*}
&(p_{21} + p_{22})(a_{11} p_{11} + a_{12} p_{12}) - (p_{11} + p_{12})(a_{21} p_{21} + a_{22} p_{22}) = 0& \\
&(p_{12} + p_{22})(b_{11} p_{11} + b_{21} p_{21}) - (p_{11} + p_{21})(b_{12} p_{12} + b_{22} p_{22}) = 0&
\end{align*}
The projective ambient space of $\mathcal{V}_X$ is three-dimensional. Its expected dimension is $1$ and its expected degree is $4$, hence it is usually a curve of degree 4.

\begin{enumerate}
\item[(a)] An example of a game that is generic in the sense that the Spohn variety is an irreducible curve of degree $4$ is given by game 114 in the \textit{Periodic Table of the $2 \times 2$ Games} of Goforth and Robinson, see~\cite{periodic-table}:
\[
\begin{pmatrix}
1 & 3\\
2 & 4
\end{pmatrix} \text{ and }
\begin{pmatrix}
4 & 1\\
2 & 3 
\end{pmatrix}
\]

\item[(b)] The game ``Prisoner's Dilemma'' can be realized by the payoff tables
\[
\begin{pmatrix}
-2 & -10\\
-1 & -5
\end{pmatrix} \text{ and }
\begin{pmatrix}
-2 & -1\\
-10 & -5
\end{pmatrix}.
\]
Its Spohn variety is a union of two irreducible curves of degree $2$ whose systems of defining polynomials are given by
\[
(p_{12} - p_{21}, p_{11}p_{21} + 9p_{21}^2  - 3p_{11}p_{22} + 5p_{21}p_{22})
\]
and
\[
(p_{11} - 5p_{22}, 9p_{12}p_{21} + 5p_{12}p_{22} + 5p_{21}p_{22} - 15p_{22}^2 ).
\]
Its real part $\mathcal{V}(\R)$ intersects the tetrahedron $\overline{\Delta}$ of all strategies, as shown in Figure~\ref{figure:prisoner}. The first irreducible component enters $\overline{\Delta}$ at the pure strategy $(1,0,0,0)$ and leaves it at $(0,0,0,1)$. The second irreducible component intersects $\overline{\Delta}$ at the other two pure strategies but enters it only later at two mixed strategies of the form $(p_{11},0,p_{12},p_{22})$ and $(p_{11},p_{21},0,p_{22})$. We will argue later in Theorem~\ref{theorem:bounds} that for generic games, including this game, the set of DE is given exactly by the common points of $\mathcal{V}(\R)$ and $\overline{\Delta}$. 
\end{enumerate}
\end{example}

\begin{figure}[h]
\includegraphics[scale=0.6]{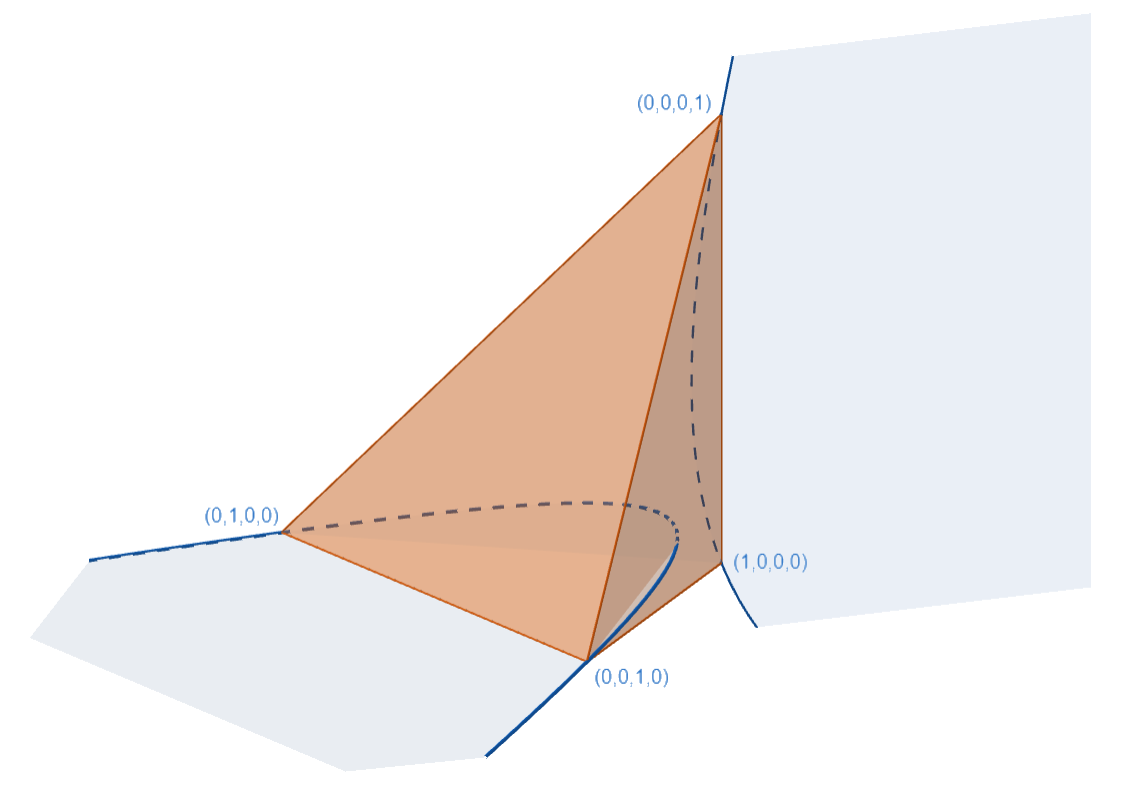}
\caption{The real part of the Spohn variety $\mathcal{V}_X$ (blue lines) of ``Prisoner's Dilemma'' and the tetrahedron $\overline{\Delta}$ of all strategies (orange). The vertices of $\overline{\Delta}$ are the pure strategies. Both components of $\mathcal{V}_X$ are plane conics. The game has two pure DE $(1,0,0,0)$ and $(0,0,0,1)$ attached to a curve segment of totally mixed DE, two isolated pure DE $(0,1,0,0)$ and $(0,0,1,0)$, and two mixed DE with one coordinate equal to $0$ which are attached to another curve segment of totally mixed DE.}\label{figure:prisoner}
\end{figure}

\section{Results for arbitrary games}\label{section:arbitraryGames}

\subsection*{Real points of the Spohn variety}

In the following, we show that the Spohn variety $\mathcal{V} = \mathcal{V}_X$ of a game $X$, that is, the variety of all complex projective points satisfying the equations of a dependency equilibrium, always has smooth points with real coordinates. This enables us to apply the following fact which, in turn, makes it possible to study the set of all DE from the point of view of complex projective varieties. As a consequence, we can apply the results of~\cite{Irem-Bernd} to the case that is actually relevant for game-theoretical applications, namely, to real points on the Spohn variety.

\begin{fact}~\cite[Theorem 2.2.9]{Mangolte} \label{fact:realpoints}
Let $V$ be an irreducible complex affine variety defined by real polynomials. If $V$ contains a smooth point with real coordinates then $V(\mathbb R)$, the set of real points of $V$, lies Zariski dense in $V$.
\end{fact}

We can apply this by taking a dense open affine subset of the ambient projective space and thereby assuming that $V = \mathcal{V}$ is affine. In general, the set of points with real coordinates of a variety can have very strange properties and might even be empty. However, the ``in particular'' of Theorem~\ref{theorem:realdense} tells us that for generic games $X$ the real part of the Spohn variety $\mathcal{V}_X$ is as large and looks structurally exactly as one would expect from looking at $\mathcal{V}_X$. Corollary~\ref{corollary:realdense} then uses this statement to infer that the set of real points on $\mathcal{V}_X$ is generically irreducible, it has (real) dimension $s - (d_1+\ldots+d_n -n)$ and degree $d_1 \cdots d_n$.

\begin{theorem}\label{theorem:realdense}
For an arbitrary game $X$, the Spohn variety $\mathcal{V} = \mathcal{V}_X$ contains a smooth point with real coordinates. In particular, if $\mathcal{V}$ is irreducible then its set of real points lies Zariski dense in $\mathcal{V}$.
\end{theorem}

\begin{proof}
For $(2 \times 2)$-games, this is just Proposition~\ref{proposition:smooth(2x2)}. 

In all other cases, we use the vector bundle parametrization $\phi: \mathcal V_X \dashrightarrow (\mathbb{P}^1)^n$ appearing in the proof of~\cite[Theorem 3.4]{Irem-Bernd} and can write
\[
\mathcal{V} = \{p \mid K_X(z) \cdot p = 0 \text{ for some } z \in (\mathbb{P}^1)^n\},
\]
where $K_X(z)$ is the Konstanz matrix of the game $X$, as defined in~\cite{Irem-Bernd}. Its entries are real numbers when $z$ has real coordinates. Note that the standing assumption of~\cite{Irem-Bernd} that the payoff tables of the game being generic is not necessary for the construction of the map. 

Take a real point $z$ in the image of $\phi$ such that $\phi^{-1}(\{z\})$ is not contained in the singular locus of $\mathcal{V}$. Now, as $\phi^{-1}(\{z\}) = \{p \mid K_X(z) \cdot p = 0\}$ is the kernel of a real matrix it contains real points and, hence, also a smooth real point of $\mathcal{V}$.
\end{proof}

Regarding game-theoretic applications, it might be interesting to find smooth real points on every irreducible component of the Spohn variety of a game, as this would prove the Zariski density of real points of the Spohn variety for arbitrary games. Consequently, the Spohn variety would, in general, structurally reflect DE. In particular, this has been proven for $(2 \times 2)$-games in \cite[Lemma 3.5]{kidambi2025elliptic}.

\begin{question}
    Does every irreducible component of the Spohn variety of a game contain a smooth real point?
\end{question}

The following is the combination of Theorem~\ref{theorem:realdense} with~\cite[Theorem 3.2]{Irem-Bernd}.

\begin{corollary}\label{corollary:realdense}
The set of real points of the Spohn variety of a game has the following properties generically: It is irreducible, of codimension $d_1+ \ldots + d_n - n$ and degree $d_1 \cdots d_n$.
\end{corollary}

\subsection*{Passing to boundary cases}

In this section, we change loosely between the projective setting of the Spohn variety and the affine setting of points in $\C^{d_1 \times \ldots \times d_n}$. Namely, if some definition depends on the choice of the affine representative of a projective point, we always take the representative whose entries sum up to $1$, as is customary in the setting of probability distributions and strategies of games.

Spohn~\cite{Spohn2003,Spohn2007} defines DE on the boundary of $\overline{\Delta}$ by using limits. The definition we use is slightly more restrictive but coincides with Spohn's for generic games. A DE is a point
$p = (p_{j_{1}\ldots j_{n}})$ in $\overline{\Delta}$ that is the limit of
a sequence $(p^{(r)})_{r\in \mathbb{N}}$ in $\Delta $ satisfying
\begin{align}\label{equation:limDE}
\lim_{r \to \infty} \mathbb E^{(i)}_k(p^{(r)}) = \lim_{r \to \infty} \mathbb E^{(i)}_{k'}(p^{(r)})
\end{align}
for all $i \in \{1,\ldots,n\}$ and $k,k' \in \{1,\ldots,d_i\}$.
In the special case of $(2 \times 2)$-games with the notation of~\cite[Example 2.2]{Irem-Bernd}, these conditions read as follows:

\begin{align*}
\lim_{r \to \infty} \frac{a_{11}p^{(r)}_{11} + a_{12}p^{(r)}_{12}}{p^{(r)}_{11}+p^{(r)}_{12}} &= \lim_{r \to \infty} \frac{a_{21}p^{(r)}_{21} + a_{22}p^{(r)}_{22}}{p^{(r)}_{21}+p^{(r)}_{22}} \\
\lim_{r \to \infty} \frac{b_{11}p^{(r)}_{11} + b_{21}p^{(r)}_{21}}{p^{(r)}_{11}+p^{(r)}_{21}} &= \lim_{r \to \infty} \frac{b_{12}p^{(r)}_{12} + b_{22}p^{(r)}_{22}}{p^{(r)}_{12}+p^{(r)}_{22}}
\end{align*}


\begin{figure}[h]
\includegraphics[scale=0.6]{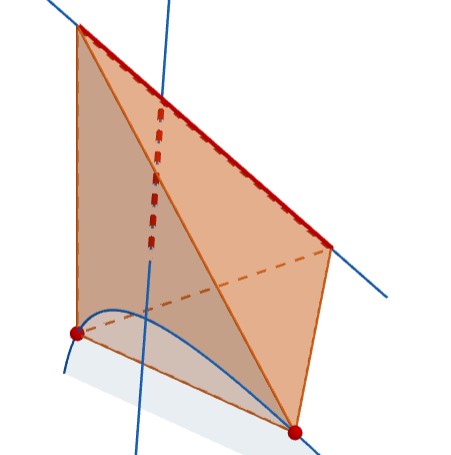}
\caption{For ``Bach or Stravinski'', the intersection (red) of the real part of the Spohn variety $\mathcal{V}$ (blue) with $\overline{\Delta}$ (orange) consists of two isolated points and two line segments (red). The isolated points are Nash equilibria but they are not DE in our sense, see Example~\ref{ex: Bach or Stravinski}.}\label{figure:DE-Bach}
\end{figure}

\begin{example}[Bach or Stravinski]\label{ex: Bach or Stravinski}
We use the payoff tables $(2,0,0,1)$ and $(1,0,0,2)$ as in~\cite[Section 3]{Spohn2007}, see also Figure~\ref{figure:DE-Bach}. The equations~(\ref{equation:limDE}) translate as
\begin{align*}
\lim_{r \to \infty} \frac{2p^{(r)}_{11}}{p^{(r)}_{11}+p^{(r)}_{12}} &= \lim_{r \to \infty} \frac{p^{(r)}_{22}}{p^{(r)}_{21}+p^{(r)}_{22}}, \\
\lim_{r \to \infty} \frac{p^{(r)}_{11}}{p^{(r)}_{11}+p^{(r)}_{21}} &= \lim_{r \to \infty} \frac{2p^{(r)}_{22}}{p^{(r)}_{12}+p^{(r)}_{22}}.
\end{align*}
 The two Nash equilibria $(1,0,0,0)$ and $(0,0,0,1)$ cannot be approximated by points in $\Delta$ satisfying the equations above. Indeed, assume to the contrary that the sequence $p^{(r)}$ converges to $(1,0,0,0)$ and the above equations are satisfied. Then, $\lim_{r \to \infty} p_{11}^{(r)} = 1$ and all other $p_{ij}^{(r)}$ converge to $0$. This means that $ \lim_{r \to \infty} \frac{2p^{(r)}_{11}}{p^{(r)}_{11}+p^{(r)}_{12}} = 2 = \lim_{r \to \infty} \frac{p^{(r)}_{22}}{p^{(r)}_{21}+p^{(r)}_{22}}$. Since all the $p_{ij}^{(r)}$ are strictly positive, it follows that $\frac{1}{2} = \lim_{r \to \infty} \frac{p^{(r)}_{21}+p^{(r)}_{22}}{p^{(r)}_{22}} = \lim_{r \to \infty} \frac{p^{(r)}_{21}}{p^{(r)}_{22}} + 1$. So $\lim_{r \to \infty} \frac{p^{(r)}_{21}}{p^{(r)}_{22}} = - \frac{1}{2}$ and hence this last sequence must get negative for $r$ sufficiently large. This is a contradiction to $p_{ij}^{(r)} > 0$.
\end{example}

We will now slightly generalize DE by also allowing limits from outside $\overline \Delta$. As a consequence, Nash equilibria will generically be DE. Then, we will describe the connection to the algebraic setting of the Spohn variety.
 Let $X$ be a game given by its payoff tensors $(X^{(i)}_{j_1\ldots j_n})$, $\mathcal{V} = \mathcal{V}_X \subseteq \mathbb{P}^{d_1\times \ldots \times d_n -1}$ the Spohn variety of $X$. Moreover, let $\mathcal{W} \subseteq \mathbb{P}^{d_1\times \ldots \times d_n -1}$ be the union of hyperplanes defined by the polynomial
\begin{align*}
s = \prod_{i,k} p_{+\ldots +k+\ldots +},
\end{align*}
where $i$ is the position in which $k$ appears.
Note, that $\mathcal{W}$ is exactly the union of those hyperplanes outside which all the conditional expected payoffs $\mathbb E_k^{(i)}(p)$ are well-defined.
For a subset $M \subseteq \C^{d_1\times \ldots \times d_n}$ (respectively $ \mathbb{P}^{d_1\times \ldots \times d_n -1}$ or $\R^{d_1\times \ldots \times d_n}$), we denote by $\overline{M}$ its usual euclidean closure, that is, the set of all points that are limits of sequences in $M$, and by $\zar{M}$ its Zariski closure, that is, the set of all points that are common zeros of all polynomials vanishing on $M$. It is a basic fact that $\overline{M} \subseteq \zar{M}$ because every algebraic set is closed in the euclidean topology.

\begin{definition}\label{def:de}
\begin{enumerate}
\item A \textit{dependency equilibrium} (DE) is a point $p$ in $\overline{\Delta}$ such that there exists a sequence $(p^{(r)})$ of complex tensors in the complement of $\mathcal{W}$ (that is, $s(p^{(r)}) \neq 0$ for all $r$) converging to $p$ such that the equations in (\ref{equation:limDE}) are satisfied.

\item An \textit{algebraic dependency equilibrium} (aDE) is a point in $\overline \Delta \cap \zar{(\mathcal{V} \setminus \mathcal{W})}$.
\end{enumerate}
\end{definition}

For all $(2 \times 2)$-games in~\cite[Section 3]{Spohn2007} (Matching pennies, Bach or Stravinski, Hawk and Dove, Prisoner's Dilemma), it can be easily checked that DE and aDE coincide.
For generic games, we can replace complex tensors by real tensors in Definition~\ref{def:de} and still deduce results analogous to the contents of this section, see Remark~\ref{remark:approximation-by-real-points}.

\begin{remark}\label{remark:definitionDE}
\begin{enumerate}

\item We could equivalently define aDE as points in $\overline{\Delta} \cap \overline{\mathcal{V} \setminus \mathcal{W}}$. This follows from the general fact that the Zariski closure and the Euclidean closure coincide for an open subset of a variety, see~\cite[Exposé XII, Prop. 2.2]{SGA1n}.

\item It follows from (1) that every aDE is a DE, because it can be approximated by a sequence in $\mathcal{V} \setminus \mathcal{W}$.

\end{enumerate}
\end{remark} 

\begin{remark}\label{remark:outsideW}
The set $\overline{\Delta} \cap (\mathcal{V} \setminus \mathcal{W})$ consists of all DE not in $\mathcal W$ because for points outside of $\mathcal{W}$ all the expressions $p_{+\ldots +k+ \ldots +}$ are non-zero.
\end{remark}

\begin{lemma}\label{lemma:DEinV(2x2)}
Every dependency equilibrium of a game $X$ lies on $\mathcal{V}_X$, the Spohn variety of $X$.
\end{lemma}

\begin{proof}
Let $p$ be a DE of $X$ and $p^{(r)}$ a sequence in $\C^{d_1 \times \ldots \times d_n} \setminus \mathcal{W}$ such that $\sum p^{(r)}_{j_1\ldots j_n} =1$, $\lim_{r \to \infty} p^{(r)} = p$, and the equations~(\ref{equation:limDE}) are satisfied. We show that $p$ satisfies every defining equation of the Spohn variety. For each player $i$ and each pair $k,k'$ of pure strategies, we get the equation
\[
p_{+\ldots+k'+\ldots+} \sum_{j_1,\ldots, j_{i-1}, j_{i+1}, \ldots, j_n} X^{(i)}_{j_1 \ldots  k \ldots j_n} p_{j_1 \ldots  k \ldots j_n} 
=
p_{+\ldots+k+\ldots+} \sum_{j_1,\ldots, j_{i-1}, j_{i+1}, \ldots, j_n} X^{(i)}_{j_1 \ldots  k' \ldots j_n} p_{j_1 \ldots  k' \ldots j_n}.
\]

If $p_{+\ldots+k+\ldots+} = 0$ then all $p_{j_1\ldots k \ldots j_n}$ are $0$ and the equation above is trivially satisfied. The same is true for $k'$ in place of $k$. If, on the other hand, $p_{+\ldots+k+\ldots+} \neq 0$ and $p_{+\ldots+k'+\ldots+} \neq 0$ then the limits $\lim_{r\to \infty} p_{+\ldots+k+\ldots+}^{(r)} = p_{+\ldots+k+\ldots+} \neq 0$ and $\lim_{r\to \infty} p_{+\ldots+k'+\ldots+}^{(r)} = p_{+\ldots+k'+\ldots+} \neq 0$ and the equation above is equivalent to the corresponding equation in~(\ref{equation:limDE}).
\end{proof}

The following is now a direct consequence of Remark~\ref{remark:definitionDE}(2) and Lemma~\ref{lemma:DEinV(2x2)}.

\begin{theorem}\label{theorem:bounds}
For the set $\text{DE}(X)$ of DE of a game $X$, there are inclusions
\begin{align*}
\overline{\Delta} \cap \zar{\mathcal{V} \setminus \mathcal{W}} = \overline{\Delta} \cap \overline{\mathcal{V} \setminus \mathcal{W}} \subseteq \text{DE}(X) \subseteq \overline{\Delta} \cap \mathcal{V},
\end{align*}
where $\mathcal{V} = \mathcal{V}_X$ is the Spohn variety of $X$.
\end{theorem}

Theorem~\ref{theorem:bounds} shows that an investigation of DE from the perspective of algebraic geometry and computational commutative algebra is indeed sensible. Every DE is a positive real point on the Spohn variety $\mathcal{V}=\mathcal{V}_X$ and every strategy that lies on an irreducible component of the Spohn variety that contains a point, for which the rational expression  $\mathbb E^{(i)}_k(p)$ of DE are well-defined, is a DE. The union of all such irreducible components is written as~$\zar{\mathcal{V} \setminus \mathcal{W}}$ in Theorem~\ref{theorem:bounds}.
These bounds are very helpful from a computational perspective because they can, given enough time, be computed by a computer algebra program such as \texttt{Macaulay2}. An explicit sample of code for $(2\times 2)$-games is shown in Example~\ref{example:code} below.

We will see in Example~\ref{example:missing-component} and Section~\ref{section:(2x2)} that neither of the two
bounds of Theorem~\ref{theorem:bounds} does necessarily agree with the set of DE. However, they do agree for generic
games as we will show in Proposition~\ref{proposition:constructible}.

\begin{example}\label{example:code}
The following is code for computations with the computer algebra software \texttt{Macaulay2} \cite{grayson2002macaulay2} of lower and upper bounds for the set of DE of a $(2\times 2)$-game.

\texttt{R} is the polynomial ring whose variables are the four entries of a joint strategy of the players. \texttt{aij} and \texttt{bij} are the players' payoffs, below for ``Prisoner's Dilemma'', which can be arbitrarily changed. \texttt{V} is the defining ideal (i.e., the set of defining polynomials) of the Spohn variety and \texttt{W} is the defining ideal of the four planes (whose union is $\mathcal{W}$) on which the rational expressions for DE are not defined.\\
\begin{small}
\begin{quote}
\begin{verbatim}
R = QQ[p11,p12,p21,p22];
a11 = -2;
a21 = -1;
a12 = -10;
a22 = -5;
b11 = -2;
b21 = -10;
b12 = -1;
b22 = -5;
V = ideal((p11+p12)*(a21*p21+a22*p22) - (p21+p22)*(a11*p11+a12*p12),
          (p11+p21)*(b12*p12+b22*p22) - (p12+p22)*(b11*p11+b21*p21));
W = ideal((p11+p12)*(p21+p22)*(p11+p21)*(p12+p22));

componentsV = minimalPrimes(V);
lowerBound = saturate(V,W);
lowerBoundEqualsSpohn = (V == lowerBound);
componentsLowerBound = minimalPrimes(lowerBound);
\end{verbatim}
\end{quote}
\end{small}
\vspace{0.3cm}

\texttt{componentsV} lists all the systems of defining equations for the irreducible components of the Spohn variety. \texttt{lowerBound} computes the defining equations for the lower bound $\zar{\mathcal{V} \setminus \mathcal{W}}$. \texttt{lowerBoundEqualsSpohn} gives \texttt{true} or \texttt{false} according to whether $\zar{\mathcal{V} \setminus \mathcal{W}} = \mathcal{V}$ or not. In particular, in case \texttt{true}, the set of DE equals the Spohn variety. Finally, \texttt{componentsLowerBound} gives the list of all the systems of defining equations for the irreducible components of $\zar{\mathcal{V} \setminus \mathcal{W}}$. Note that this is a sublist of \texttt{componentsV}.
\end{example}

\begin{example}\label{example:missing-component}
We consider the game
\[
\begin{pmatrix}
1 & 1\\
2 & -3
\end{pmatrix} \text{ and }
\begin{pmatrix}
-1 & 0\\
0 & 2
\end{pmatrix}
\]
from the class of games introduced in Example~\ref{example:special}. The real part of its Spohn variety is shown in Figure~\ref{figure:missing-component}. It is the union of two irreducible plane conics which intersect the set of all strategies $\overline{\Delta}$ exactly in the four pure strategies. 

One of the irreducible components is contained in the exceptional planes whose union is $\mathcal{W}$. Hence, $\zar{\mathcal{V} \setminus \mathcal{W}} \neq \mathcal{V}$. The two pure strategies $(1,0,0,0)$ and $(0,0,1,0)$ lie in $\zar{\mathcal{V} \setminus \mathcal{W}} \cap \overline{\Delta}$ and are, therefore, DE. The other two pure strategies might be DE because they lie in $\mathcal{V} \cap \overline{\Delta}$ but further considerations are needed.
\end{example}

\begin{figure}[h]
\includegraphics[scale=0.65]{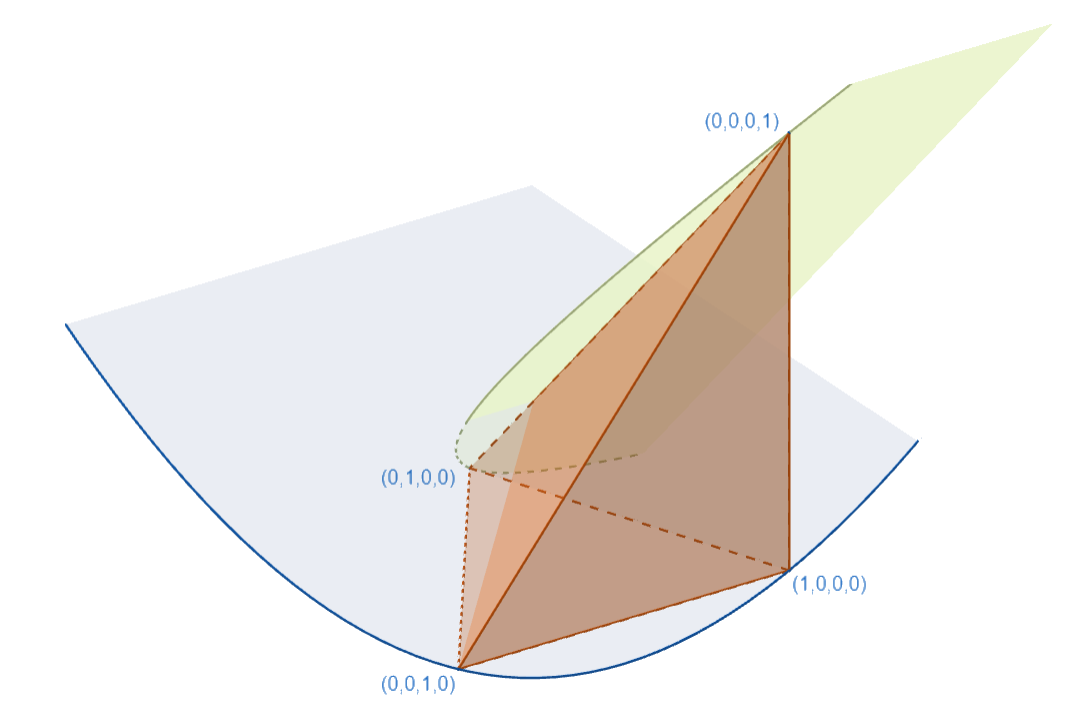}
\caption{The real part of the Spohn variety of the $(2\times 2)$-game in Example~\ref{example:missing-component}. The lower bound $\zar{\mathcal{V} \setminus \mathcal{W}}$ is given by the blue line. But the Spohn variety has another irreducible component (green).}\label{figure:missing-component}
\end{figure}

\begin{remark}\label{remark:approximation-by-real-points}
    Theorem~\ref{theorem:bounds} is still valid for generic games when one replaces complex tensors by real tensors in Definition~\ref{def:de}(1). In other words, for generic games $X$, the following three statements are equivalent for a strategy $p$:
    \begin{enumerate}
        \item[(a)] $p$ is a DE of $X$ in the sense of Definition~\ref{def:de}(1).
        \item[(b)] There exists a sequence $(p^{(r)})$ of real tensors in the complement of $\mathcal{W}$ (that is, $s(p^{(r)}) \neq 0$ for all $r$) converging to $p$ such that the equations in (\ref{equation:limDE}) are satisfied.
        \item[(c)] $p$ lies on the Spohn variety of $X$.
    \end{enumerate}
\end{remark}

\begin{proof}
    For generic games, the Spohn variety $\mathcal{V}$ is irreducible and not contained in $\mathcal{W}$ by~\cite[Theorem 3.2]{Irem-Bernd} and Proposition~\ref{proposition:constructible}. Hence, (a) and (c) are equivalent. Moreover, the implication from (b) to (a) is trivial.

    For the remaining implication from (a) to (b), it is sufficient to show that $\overline{\mathcal{V} \setminus \mathcal{W}} \cap \overline{\Delta}$ (which coincides with $\text{DE}(G)$ for generic games) is contained in the set $\text{DE}_\R(G)$ of all points satisfying the statement in (b). For this, it suffices to show the equality
    \begin{equation}\label{equation:real}
        \overline{\mathcal V \setminus \mathcal W} \cap \R^s = \overline{(\mathcal V \setminus \mathcal W) \cap \R^s},
    \end{equation}
    where $s$ is the dimension of the ambient space. Indeed, if we intersect the right side of~(\ref{equation:real}) with $\overline{\Delta}$ then the resulting set is contained in $\text{DE}_\R(G)$, because $(\mathcal{V} \setminus \mathcal{W}) \cap \R^s$ consists of real tensors satisfying the defining rational equations for DE.

   To show~(\ref{equation:real}), note that $\overline{\mathcal{V}(\R) \setminus \mathcal{W}(\R)}$ is a real variety contained in the real variety $\overline{\mathcal{V} \setminus \mathcal{W}} \cap \R^s$ and both have the same dimension because $\mathcal{V}(\R)$ lies Zariski dense in $\mathcal{V}$ by Theorem~\ref{theorem:realdense}. Therefore, they are the same. Also, $\mathcal{V}(\R) \setminus \mathcal{W}(\R) = (\mathcal{V} \setminus \mathcal{W}) \cap \R^s$ is always true for set-theoretical reasons and, therefore,~(\ref{equation:real}) follows.
\end{proof}

\begin{remark}
    That~(\ref{equation:real}) is not true in general for an algebraic variety $\mathcal{V}$ and a union of hyperplanes $\mathcal{W}$ can be seen by the following example: Let $\mathcal{V}$ be an irreducible complex surface of degree greater than $1$ whose real part is just a line. For instance, consider the surface in complex affine $3$-space defined by the equation
    \[
    2x^2+2xy+y^2-2xz-2yz+z^2-2x+1 = 0.
    \]
    Its set of real points is the line defined by
    \[
    y-z+1 = x-1 = 0.
    \]

    If now $\mathcal{W}$ is a hyperplane containing this line then $\overline{\mathcal{V} \setminus \mathcal{W}} = \mathcal{V}$ and hence $\overline{\mathcal{V} \setminus \mathcal{W}} \cap \R^s = \mathcal{V}(\R)$ which is the line. But $(\mathcal{V} \setminus \mathcal{W}) \cap \R^s$ is the empty set, so~(\ref{equation:real}) does not hold.
\end{remark}

\begin{remark}\label{remark:constructible}
In view of Theorem~\ref{theorem:bounds}, it might be interesting to understand for which games the equality $\zar{\mathcal{V} \setminus \mathcal{W}} = \mathcal{V}$ holds. Indeed, in these cases, the set of DE coincides with $\overline{\Delta} \cap \mathcal{V}$.
In general, $\zar{\mathcal{V} \setminus \mathcal{W}} = \mathcal{V}$ if and only if no irreducible component of $\mathcal{V}$ is contained in $\mathcal{W}$. 

A game $X$ (of a fixed size) given in normal form by its payoff tables $(X^{(i)}_{j_1\ldots j_n})_{(j_1,\ldots,j_n) \in d_1 \times \ldots \times d_n}$, $ i = 1,\ldots,n$, is nothing else than a point in the affine space $\R^{n\times d_1\times \ldots \times d_n} \subseteq \C^{n\times d_1\times \ldots \times d_n}$. 
So we could ask for a structural result on the set of all affine points $X \in \R^{n\times d_1\times \ldots \times d_n}$ such that 
$\mathcal{V}_X \setminus \mathcal{W}$ lies (Zariski) dense inside $\mathcal{V}_X$.
\end{remark}

\begin{proposition}\label{proposition:constructible}
The set $\mathcal{X}$ of all games $X\in \R^{n\times d_1\times \ldots \times d_n}$ such that $\zar{\mathcal{V}_X \setminus \mathcal{W}} = \mathcal{V}_X$ (which implies that the set of DE of $X$ coincides with the non-negative real part of $\mathcal{V}_X$) is the complement of an affine variety in $\R^{n\times d_1\times\ldots \times d_n}$. In particular, it is a dense subset of $\R^{n\times d_1\times\ldots \times d_n}$.
\end{proposition}

\begin{proof}
We use~\cite[Prop. 9.5.3]{EGAIV-part3}. We consider the $\C$-algebras 
\[
\mathcal{A} = \C[X^{(i)}_{j_1\ldots j_n} \mid i \in \{1,\ldots,n\}, (j_1,\ldots,j_n) \in [d_1] \times \ldots \times [d_n]]
\]
and
\[
\mathcal{B} = \frac{\C[X^{(i)}_{j_1\ldots j_n}, p_{j_1,\ldots,j_n} \mid i \in \{1,\ldots,n\}, (j_1,\ldots,j_n) \in [d_1] \times \ldots \times [d_n]]}{\mathcal{I}_X},
\]
where the $X^{(i)}_{j_1\ldots j_n}$ are variables for the entries of the payoff tables, the $ p_{j_1,\ldots,j_n}$ are variables for the entries of a strategy and $\mathcal{I}_X$ is the ideal generated by the polynomials of the defining equation of the Spohn variety (with variable payoff tables). A $\C$-algebra homomorphism $\mathcal{A} \to \mathcal{B}$ is given by composing the canonical inclusion 
\[
\mathcal{A} \to \C[X^{(i)}_{j_1\ldots j_n}, p_{j_1,\ldots,j_n} \mid i \in \{1,\ldots,n\}, (j_1,\ldots,j_n) \in [d_1] \times \ldots \times [d_n]]
\]
with the projection
\[
\C[X^{(i)}_{j_1\ldots j_n}, p_{j_1,\ldots,j_n} \mid i \in \{1,\ldots,n\}, (j_1,\ldots,j_n) \in [d_1] \times \ldots \times [d_n]] \to \mathcal{B}.
\]
From this, we get an induced finitely presented morphism $f$ of affine schemes $T := \Spec(\mathcal{B}) \to \Spec(\mathcal{A}) =: S$ and the fibre (intersected with $\Max(\mathcal{B})$) of a point $X \in \C^{n\times d_1\times \ldots \times d_n} = \Max(\mathcal{A})$ is the Spohn variety of $X$. So, $f$ parametrizes the family of all Spohn varieties along the set of all games of the fixed size $n\times d_1 \times \ldots \times d_n$.

We derive the closed subscheme $\mathcal W$ of $X$ from the projection $\mathcal{B} \to \mathcal{B}/(s)$, where \[s = \prod_{i,k} p_{+\ldots +k+\ldots +}.\]
We set $Z = T \setminus \mathcal{W}$ and $Z' = T$. Now, for any $X \in \Max(\mathcal{A})$, $Z' \cap f^{-1}(X) = \mathcal{V}_X$ and $Z \cap f^{-1}(X) = \mathcal{V}_X \setminus \mathcal{W}$. With this in mind, the set $\mathcal{X}$ of the games in question is a constructible set by~\cite[Prop. 9.5.3]{EGAIV-part3}, using exactly the same notation. 

It is immediate from the computations of the Jacobian matrix of $\mathcal{V}_X$ in Remark~\ref{remark:Jacobian} that the set of all games $X$ such that $\mathcal{V}_X$ intersects all irreducible components of $\mathcal{W}$ transversally is Zariski open. But this set is clearly contained in $\mathcal{X}$. So $\mathcal{X}$ is a constructible set in an affine space containing a Zariski open set and it is therefore itself Zariski open, that is, the complement of an affine variety. 
\end{proof}

By Definition~\ref{def: Spohn variety}, every pure strategy lies on the Spohn variety. An immediate observation and an extension of \cite[Theorem 3.2]{Irem-Bernd} is that every pure
strategy is generically a DE.

\begin{problem}\label{problem:constructible}
For fixed $n$ and $d_1,\ldots,d_n$, describe the Zariski open set $\mathcal{X}$ (see Prop.~\ref{proposition:constructible}), that is, give the defining equations of its complement.
\end{problem}

The solution of this problem for $(2\times 2)$-games is Theorem~\ref{theorem:2x2}.

\subsection*{Nash equilibria and the Spohn variety}
While Theorem~\ref{theorem:NEonSpohn} shows that every Nash equilibrium lies on the Spohn variety, Example~\ref{example:NEnotDE} shows that not
every Nash equilibrium is a DE. However, Corollary~\ref {corollary:NEisDEgeneric} gives a criterion for when a Nash equilibrium is a DE which is generically satisfied. Note that the totally mixed Nash equilibria are exactly the positive real points that lie both on the Spohn variety and on the Segre variety by a result of Sturmfels and the first author~\cite[Theorem 3.2]{Irem-Bernd}. In particular, totally mixed Nash equilibria are always DE.

By $\Delta_{d-1}$ we denote the $(d-1)$-dimensional standard simplex. It can be considered as a subset of a $d$-dimensional affine space or of a $(d-1)$-dimensional projective space.

\begin{figure}[h]
\includegraphics[scale=0.6]{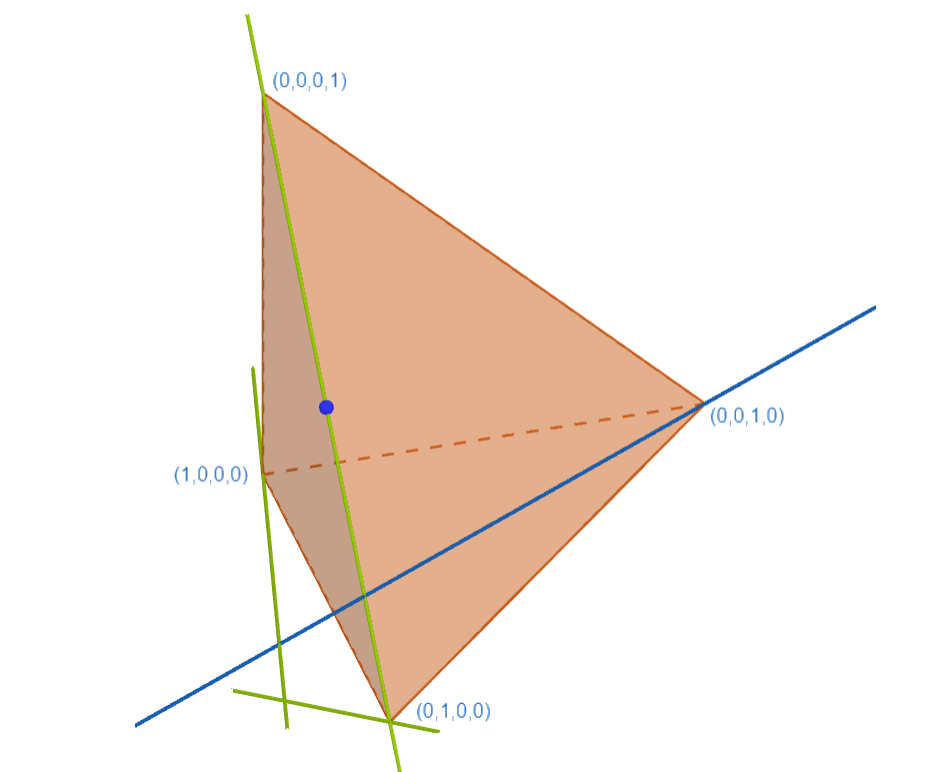}
\caption{
 The real part of the Spohn variety consisting of four lines for the game with payoff tables
$
\begin{pmatrix}
-1 & -1\\
-2 & 3
\end{pmatrix} \text{ and }
\begin{pmatrix}
-3 & 0\\
0 & 0
\end{pmatrix}$.
The lower bound $\zar{\mathcal{V} \setminus \mathcal{W}}$ for DE equals the blue component. By Example~\ref{example:NEnotDE}, the blue point is a Nash equilibrium but not a dependency equilibrium. This is only possible because it lies on one of the green components of $\mathcal{V}$, see Corollary~\ref{corollary:NEisDEgeneric}. These are exactly the components of $\mathcal{V}$ that are contained in $\mathcal{W}$.}\label{figure:NE}
\end{figure}

\begin{lemma}\label{lemma:rank-one}
A rank-one point $p=(p_1^{(1)},\ldots,p_{d_1}^{(1)}, \ldots \ldots,p_1^{(n)},\ldots,p_{d_n}^{(n)}) \in \Delta_{d_1 - 1} \times \ldots \times \Delta_{d_n - 1} \subseteq \overline{\Delta}$ lies on the Spohn variety of the game $X$ if and only if 
\begin{equation}\label{equation:non-zero}
\sum_{j_1 = 1}^{d_1} \ldots \sum_{j_{i-1} = 1}^{d_{i-1}} \sum_{j_{i+1} = 1}^{d_{i+1}} \ldots \sum_{j_n = 1}^{d_n} \left(X^{(i)}_{j_1\ldots j_{i-1} k j_{i+1} \ldots j_n} - X^{(i)}_{j_1\ldots j_{i-1} k' j_{i+1} \ldots j_n}\right) p_{j_1}^{(1)} \cdots p_{j_{i-1}}^{(i-1)} p_{j_{i+1}}^{(i+1)} \cdots p_{j_n}^{(n)} = 0
\end{equation}
for all $i \in \{1,\ldots,n\}$ and all $k,k' \in \{1,\ldots,d_i\}$ with $p_k^{(i)} >0$ and $p_{k'}^{(i)} > 0$.
\end{lemma}

\begin{proof}
To recover the notation for general tensors on the Spohn variety, we set $p_{j_1 \ldots j_n} = p_{j_1}^{(1)} \cdots p_{j_n}^{(n)}$. Note that $p_{+\ldots+k+\ldots+} = p_k^{(i)}$ for all $i$ and $k$. Let $i \in \{1,\ldots,n\}$ and $k,k' \in \{1,\ldots,d_i\}$. If $p_k^{(i)} =0$ or $p_{k'}^{(i)}=0$ then the corresponding defining equation for the Spohn variety is trivial. So, assume that both are non-zero. Then the corresponding defining equation of the Spohn variety boils down for rank-one tensors to (\ref{equation:non-zero}), as noted in the proof of~\cite[Theorem 3.2]{Irem-Bernd}.
\end{proof}

\begin{theorem}\label{theorem:NEonSpohn}
Every Nash equilibrium of a game $X$ lies on $\mathcal{V}_X$, the Spohn variety of $X$.
\end{theorem}

\begin{proof}
We use the notation of Lemma~\ref{lemma:rank-one} for the given Nash equilibrium $p$ of $X$.  If $p_k^{(i)} =0$ or $p_{k'}^{(i)}=0$ then the corresponding defining equation for the Spohn variety is trivial. In the case that both are non-zero, it is a simple computation that the equations of~\cite[Theorem 6.6]{solving} are equivalent to (\ref{equation:non-zero}), and hence $p$ lies on $\mathcal{V}_X$.
\end{proof}

\begin{corollary}\label{corollary:NEisDEgeneric}
If the variety $\mathcal{W}$ contains no irreducible component of the Spohn variety $\mathcal{V}_X$ of a game $X$ then every Nash equilibrium of $X$ is a dependency equilibrium of $X$.
In particular, when $X$ is generic, every Nash equilibrium of $X$ is a dependency equilibrium of $X$.
\end{corollary}

\begin{proof}
This follows from Theorem~\ref{theorem:NEonSpohn} in combination with Remark~\ref{remark:constructible}.
\end{proof}

\subsection*{Existence of pure and totally mixed DE}

We now employ an easy idea from algebraic geometry to find large collections of games that admit infinitely many totally mixed DE. Our method will then guarantee, in addition, the existence of a pure DE.

\begin{remark}\label{remark:Jacobian}
Let $\mathcal{V}$ be the Spohn variety of a game given by its payoff tensors $(X^{(i)}_{j_1\ldots j_n})$. The rows of the Jacobian matrix $\mathcal{J}(p)$ of $\mathcal{V}$ at a point $p = (p_{j_1\ldots j_n})$ has row indices $(i,k,k')$ with $1 \leq i \leq n$ and $1 \leq k < k' \leq d_i$ (corresponding to the defining equations of $\mathcal{V}$) and column indices $(r_1,\ldots,r_n)$ where $1 \leq r_j \leq d_j$ (corresponding to the partial derivative at the variable $p_{r_1\ldots r_n}$). The entry with this row and column index is $0$ unless $r_i \in \{k,k'\}$. In these cases, it is
\[
p_{+\ldots+k+\ldots+} X^{(i)}_{r_1\ldots k' \ldots r_n} - \sum_{ j_1,\ldots, j_{i-1}, j_{i+1}, \ldots, j_n } X^{(i)}_{j_1\ldots k \ldots j_n} p_{j_1\ldots k \ldots j_n}
\]
and
\[
\sum_{ j_1,\ldots, j_{i-1}, j_{i+1}, \ldots, j_n } X^{(i)}_{j_1\ldots k' \ldots j_n} p_{j_1\ldots k' \ldots j_n} - p_{+\ldots+k'+\ldots+} X^{(i)}_{r_1\ldots k \ldots r_n},
\]
respectively. 

If $p$ is now the pure strategy that has a $1$ at position $(j_1,\ldots,j_{i-1},k',j_{i+1},\ldots,j_n)$ and $0$ elsewhere then an entry of $\mathcal{J}(p)$ is $0$ unless $r_i = k$. The entry in row $(i,k,k')$ and column $(r_1,\ldots,r_{i-1},k,r_{i+1},\ldots, r_n)$ equals
\[
X^{(i)}_{j_1\ldots k' \ldots j_n} - X^{(i)}_{r_1\ldots k \ldots r_n}. 
\]
Recall that the pure strategy $p$ always lies on $\mathcal{V}$ by definition.
\end{remark}

\begin{theorem}\label{Theorem:tangent}
    Suppose that the Spohn variety $\mathcal{V}$ of a game $X$ is smooth at the pure strategy $p$, that is, its Jacobian $\mathcal{J}(p)$ at $p$ (see Remark~\ref{remark:Jacobian}) has maximal possible rank (the codimension of $\mathcal{V}$). Then the following are equivalent
    \begin{enumerate}
        \item Every (Euclidean) neighborhood of $p$ contains a totally mixed dependency equilibrium of $X$.
        \item The kernel of $\mathcal{J}(p)$ contains a vector all of whose entries are positive real numbers.
    \end{enumerate}

In this case, $p$ is a pure dependency equilibrium of $X$.

\end{theorem}

\begin{proof}
    The kernel of $\mathcal{J}(p)$ shifted by $p$, that is, $\ker \mathcal{J}(p) + p$ is the tangent space of $\mathcal{V}$ at $p$. As the column $(r_1,\ldots,k',\ldots,r_n)$ of $\mathcal{J}(p)$ consists of zeros, $\ker \mathcal{J}(p)$ contains a vector with all entries positive if and only if this tangent space does. This is the case if and only if the tangent space has non-empty intersection with $\Delta$ which is true exactly when $\mathcal{V} \cap U \cap \Delta \neq \emptyset$ for every (euclidean) neighbourhood $U$ around $p$. This, in turn, is equivalent to (1) in the statement of the proposition.

    As $\mathcal{V}\cap U \cap \Delta \neq \emptyset$ for every neighbourhood $U$ of $p$, a whole irreducible component of $\mathcal{V}$ containing $p$ passes through $\Delta$ and is therefore not contained in $\mathcal{W}$. It follows by Theorem~\ref{theorem:bounds} that $p$ is a DE.
\end{proof}

 Based on Theorem~\ref{Theorem:tangent}, we will later give a very explicit form of this criterion for $(2\times 2)$-games involving only equalities and inequalities in the entries of the payoff tables, see Example~\ref{example:Jac(2x2)}. Figure~\ref{figure:prisoner} shows the Spohn variety of an example of a game where this criterion applies for the pure strategies $(1,0,0,0)$ and $(0,0,0,1)$, but it fails for the other two pure strategies, although these are also DE.

\begin{remark}\label{remark:tangent}
\begin{enumerate}
    \item Clearly, the method of the proof of Theorem~\ref{Theorem:tangent} works for arbitrary $p \in \overline \Delta$ that is a smooth point of $\mathcal{V}$: Every Euclidean neighborhood of $p$ contains a DE of $X$ if and only if $\ker \mathcal{J}(p) + p$ contains a positive vector.
    \item What is more, one could still use the same ideas to show that $p$ is a DE of $X$ by checking that $\ker \mathcal{J}(p) + p$ is not contained in $\mathcal{W}$.
    \item Even if $\ker \mathcal{J}(p) + p \subseteq \mathcal{W}$ it could still be possible to find out that $p$ is a DE of $X$ by exhibiting that an irreducible component of $\mathcal{V}$, containing $p$ as a smooth point, is not linear, for instance, by studying second order derivatives of $\mathcal{V}$ around $p$.
\end{enumerate}
\end{remark}

\section{General results for $(2 \times 2)$-games}\label{section:(2x2)}

Next, we show that at least one of the four pure strategies is always a smooth point on the Spohn variety of a game. 

\begin{proposition}\label{proposition:smooth(2x2)}

The Spohn variety of the $(2\times 2)$-game 
\[
\begin{pmatrix}
a_{11} & a_{12}\\
a_{21} & a_{22} 
\end{pmatrix} \text{ and }
\begin{pmatrix}
b_{11} & b_{12}\\
b_{21} & b_{22} 
\end{pmatrix}
\]
contains $p = (1,0,0,0)$ as a smooth point if and only if one of the following holds:
\begin{enumerate}
    \item $a_{11} \neq a_{21}$ and $b_{11} \neq b_{12}$
    \item $a_{11} = a_{21}$, $a_{11} \neq a_{22}$ and $b_{11} \neq b_{12}$
    \item $b_{11} = b_{12}$, $b_{11} \neq b_{22}$ and $a_{11} \neq a_{21}$.
\end{enumerate}

One of the four pure strategies is always a smooth point on the Spohn variety.
\end{proposition}

\begin{proof}

The Jacobian of the Spohn variety at $p$ is
\[ \mathcal{J}(p) = 
\begin{pmatrix}
0 & a_{11} - a_{21} & 0 & a_{11} - a_{22}\\
0 &  0 & b_{11}-b_{12} & b_{11}-b_{22}
\end{pmatrix}.
\]

By considering analogous characterizations for the other pure strategies, it can be easily checked that at least one of the four pure strategies is a smooth point, unless the pay-off tables are of the form shown in Theorem~\ref{theorem:dimension(2x2)}(3)(a) in which case the Spohn variety is smooth anyway.
\end{proof}

$(2\times 2)$-games split up into distinct classes concerning the geometric structure of their sets of DE. An understanding of these classes is based on the considerations of Theorem~\ref{theorem:dimension(2x2)} below and its proof which give, from this viewpoint, a full and sensible characterization of $(2\times 2)$-games. Cases (3)(c) and (3)(d) deserve further investigation concerning the number of irreducible components of the Spohn variety (which all are curves in these cases) and their individual degrees. This is studied in \cite{kidambi2025elliptic}.

\begin{theorem}\label{theorem:dimension(2x2)}
There are the following cases for the Spohn variety $\mathcal{V}$ of the $(2\times 2)$-game 
\[
\begin{pmatrix}
a_{11} & a_{12}\\
a_{21} & a_{22} 
\end{pmatrix} \text{ and }
\begin{pmatrix}
b_{11} & b_{12}\\
b_{21} & b_{22} 
\end{pmatrix}.
\]

\begin{enumerate}
\item Both payoff tables are constant in which case $\mathcal{V} = \mathbb{P}^3$.
\item Exactly one of the payoff tables is constant.
\begin{itemize}
\item[(a)] $\mathcal{V}$ is an irreducible surface of degree $2$ and so is its real part.
\item[(b)] $\mathcal{V}$ is the union of two distinct planes and so is its real part.
\end{itemize}
\item None of the payoff tables is constant.
\begin{itemize}
\item[(a)] The payoff tables are of the form
\[
\begin{pmatrix}
a & c\\
a & c 
\end{pmatrix} \text{ and }
\begin{pmatrix}
b & b\\
d & d 
\end{pmatrix}
\]
with $a\neq c$ and $b \neq d$. Then $\mathcal{V}$ is an irreducible surface of degree $2$, the Segre variety in $\mathbb{P}^3$, and so is its real part.
\item[(b)] $\mathcal{V}$ is the intersection of two varieties which are both the union of two distinct planes.
\begin{itemize}
\item $\mathcal{V}$ is the union of a plane and a line not contained in the plane, and so is its real part.
\item $\mathcal{V}$ is the union of two lines with empty intersection, and so is its real part.
\end{itemize}
\item[(c)] $\mathcal{V}$ is the intersection of a degree $2$ irreducible surface and a variety which is the union of two distinct planes and, hence, $\mathcal{V}$ is the union of two distinct curves.
\item[(d)] $\mathcal{V}$ is the intersection of two distinct degree $2$ surfaces and, hence, it is the union of curves. 
\end{itemize}
\end{enumerate}
This last case (3)(d) happens if and only if none of the following assertions holds:
\begin{center}
(i) $a_{11} = a_{21}$ and $a_{12} = a_{22}$ and $b_{11} = b_{12}$ and $b_{21} = b_{22}$,
\begin{tabular}{c c}
(ii) $a_{11} = a_{21}$ and $a_{11} = a_{22}$, & (iii) $a_{12} = a_{21}$ and $a_{12} = a_{22}$, \\
(iv) $a_{11} = a_{22}$ and $a_{12} = a_{22}$, & (v) $a_{11} = a_{21}$ and $a_{12} = a_{21}$, \\
(vi) $b_{11} = b_{12}$ and $b_{11} = b_{22}$, & (vii) $b_{21} = b_{12}$ and $b_{21} = b_{22}$, \\
(viii) $b_{11} = b_{22}$ and $b_{21} = b_{22}$, & (ix) $b_{11} = b_{12}$ and $b_{21} = b_{12}$.
\end{tabular}
\end{center}
\end{theorem}

\begin{proof}
After easy computations the two defining polynomials of the Spohn variety are represented as
\[
f_a = p_{11}(p_{21}(a_{11}-a_{21}) + p_{22}(a_{11} - a_{22})) + p_{12}(p_{21}(a_{12} - a_{21}) + p_{22}(a_{12} - a_{22}))
\]
and
\[
f_b = p_{11}(p_{12}(b_{11}-b_{12}) + p_{22}(b_{11}-b_{22})) + p_{21}(p_{12}(b_{21}-b_{12}) + p_{22}(b_{21}-b_{22})).
\]
Only when (i) (or, equivalently, (3)(a)) is satisfied can these two polynomials agree (up to a constant multiplicative factor). So this is the only case in which both polynomials define the same variety.

The statements (ii)-(v) exactly characterize the cases when $f_a$ is not irreducible. Indeed, in all other cases, $f_a$ is a primitive linear polynomial considered as polynomial in one of the four variables over the polynomial ring in the other three variables. For instance, the negations of (ii) and (iii) imply that $f_a$ is a non-constant linear polynomial in $p_{11}$ with non-zero constant term, and the negations of (iv) and (v) imply that the coefficients (in $\C[p_{21},p_{12},p_{22}]$) of this linear polynomial are coprime.
Analogously, (vi)-(ix) exactly characterize the cases when $f_b$ is not irreducible.
These observations give the characterization of case (3)(d).

A careful but easy analysis of the statements (i)-(ix) and of the irreducible factors of the representations of $f_a$ and $f_b$ above yields the other cases of the theorem.
\end{proof}

\begin{problem}
    Investigate further the cases (3)(c) and (3)(d) of Theorem~\ref{theorem:dimension(2x2)}. Characterize the number and the geometric structure of the irreducible components of the Spohn variety by algebraic relations on the payoffs, in these cases.
\end{problem}

We will now present a class of $(2 \times 2)$-games that will show us the following:
\begin{enumerate}
\item Not every DE is an aDE (not even for $(2 \times 2)$-games).
\item The converse of Lemma~\ref{lemma:DEinV(2x2)} is not true, that is, there might exist points of $\overline{\Delta} \cap \mathcal{V}$ that are not DE.
\end{enumerate}

\begin{example}\label{example:special}
Consider, for arbitrary $a_1,a_2,b_1,b_2,c \in \R$, the game $X$ with payoff tables
\[
\begin{pmatrix}
a_{1} & a_{1}\\
2 a_{1} & a_{2} 
\end{pmatrix} \text{ and }
\begin{pmatrix}
b_{1} & c\\
c & b_{2} 
\end{pmatrix}.
\]
The Spohn variety of this game has at least two irreducible components and, when the above numbers are chosen generically then the irreducible components are given by the two minimal prime ideals of the ideal of $\mathcal{V}$:
\begin{align*}
P &= (p_{11} + p_{12}, \ b_1 p_{12}^2 + b_1 p_{12} p_{22} + b_2 p_{21} p_{22} - b_2 p_{12} p_{22} - c p_{12}^2 - c p_{21} p_{22}) \\
Q &= (a_1 p_{21} - a_1 p_{22} + a_2 p_{22}, \ b_1 p_{11} p_{12} + b_1 p_{11} p_{22} - b_2 p_{11} p_{22} - b_2 p_{21} p_{22} - c p_{11} p_{12} + c p_{21} p_{22})
\end{align*}
Note that the irreducible components of $\mathcal{V}$ whose union is $V(P)$ are contained in $W$ and hence $\overline{\mathcal{V} \setminus \mathcal{W}} = V(Q)$. Here, $V(P)$ denotes the variety defined by $P$, that is, the set of all common zeros of polynomials in $P$.
\end{example}

\begin{proposition}
If $a_1 \neq 0$ then the pure strategy
\[
\begin{pmatrix}
0 & 0 \\
1 & 0 
\end{pmatrix}
\]
is a DE but not an aDE of the game $X$ of Example~\ref{example:special}.
\end{proposition}

\begin{proof}
As seen in Example~\ref{example:special}, $\overline{\Delta}\cap \overline{\mathcal V \setminus \mathcal{W}} = \overline{\Delta} \cap V(Q)$ for such a particular game $X$. But the pure strategy $p$ from the statement of the proposition is not on $V(Q)$ unless $a_1 = 0$. Therefore, it is not an aDE of $X$.

On the other hand, the sequence
\[
\begin{pmatrix}
\frac{1}{2r} & \frac{1}{2r} \\
1 - \frac{1}{r} & 0 
\end{pmatrix},
\]
with $r \in \Z_{>0}$, has no element in $\mathcal{W}$, converges to $p$ and satisfies the equations~(\ref{equation:limDE}). Hence, $p$ is a DE of $X$.
\end{proof}

\begin{proposition}\label{proposition:SpohnnotDE}
If $a_2 \neq 0$ and $c = b_2$ then the strategy
\[
\begin{pmatrix}
0 & 0 \\
1/2 & 1/2
\end{pmatrix}
\]
lies on the Spohn variety of the game $X$ of Example~\ref{example:special} but is not a DE of $X$.
\end{proposition}

\begin{proof}
It is an easy computation that, under the assumption $c = b_2$, the strategy $p$ of the statement of the proposition lies on $V(P)$ which is a subset of the Spohn variety. 

Assume to the contrary that $p$ is a DE of $X$. In particular, it follows from the equations~(\ref{equation:limDE}) that 
\[
a_1 = \lim_{r \to \infty} \frac{a_{1}p_{11}^{(r)} + a_{1}p_{12}^{(r)}}{p_{11}^{(r)} + p_{12}^{(r)}} =  \lim_{r \to \infty} \frac{2 a_{1}p_{21}^{(r)} + a_{2}p_{22}^{(r)}}{p_{21}^{(r)} + p_{22}^{(r)}} = a_1 + \frac{1}{2} a_2
\]
which is a contradiction as $a_2 \neq 0$. 
\end{proof}

Next, we give an example of a Nash equilibrium that is not a DE.

\begin{example}\label{example:NEnotDE}
We go to the situation of Proposition~\ref{proposition:SpohnnotDE}, that is, we consider the $(2 \times 2)$-game with payoff tables
\[
\begin{pmatrix}
a_{1} & a_{1}\\
2 a_{1} & a_{2} 
\end{pmatrix} \text{ and }
\begin{pmatrix}
b_{1} & b_2\\
b_2 & b_{2}
\end{pmatrix},
\]
where $a_2 \neq 0$, and the strategy
\[
p = \begin{pmatrix}
0 & 0 \\
1/2 & 1/2
\end{pmatrix}
= (0,1) \otimes \left(\frac{1}{2}, \frac{1}{2}\right)\in \Delta_1 \times \Delta_1.
\]
We assume in addition that $a_2 > 0$.
We already know that $p$ is not a DE of the game but it lies on its Spohn variety. 
We use the notation of~\cite[Theorem 6.4]{solving}, that is, $p_1^{(1)} = 0$, $p_2^{(1)} = 1$, and $p_1^{(2)} = p_2^{(2)} = \frac{1}{2}$. For $p_k^{(i)} \neq 0$, it is a simple computation that the corresponding equations from~\cite[Theorem 6.6]{solving} are satisfied. 

As $p_1^{(1)} = 0$, we have to check that the parenthesis expression on the right of the equations in~\cite[Theorem 6.4]{solving} is non-negative. This expression is
\[
\frac{1}{2} (2a_1 + a_2) - (\frac{1}{2} a_1 + \frac{1}{2} a_1) = \frac{1}{2} a_2
\]
which is positive by assumption. So $p$ is a Nash equilibrium.
\end{example}

In the following example, we work out and apply Theorem~\ref{Theorem:tangent} in the case of $(2\times 2)$-games.

\begin{example}\label{example:Jac(2x2)}

Let us go back to the set-up of Proposition~\ref{proposition:smooth(2x2)} and set $p = (1,0,0,0)$. In the situations (2) and (3), a vector $(w,x,y,z)$ in $\ker \mathcal{J}(p)$ must satisfy $z = 0$ and hence, by Theorem~\ref{Theorem:tangent}, there is a neighbourhood of $p$ not containing a totally mixed DE of the game.

In the situation (1), that is, $a_{11} \neq a_{21}$ and $b_{11} \neq b_{12}$, the real vectors in $\ker \mathcal{J}(p)$ are exactly of the form
\[
\left(w,-\frac{a_{11} - a_{22}}{a_{11} - a_{21}} z, -\frac{b_{11}-b_{22}}{b_{11}-b_{12}} z, z\right)
\]
with $w,z$ real numbers. Such a vector can have only positive entries for some choices of $w$ and $z$ if and only if
\begin{equation}\label{equation:(2x2)tangent}
    (a_{11}-a_{22})(a_{11}-a_{21}) < 0 
\text{ and }
(b_{11}-b_{22})(b_{11}-b_{12}) < 0.
\end{equation}
So, these are exactly the cases in which every neighbourhood around $p$ contains totally mixed DE of the game by Theorem~\ref{Theorem:tangent}. In particular, if (\ref{equation:(2x2)tangent}) is satisfied then $p = (1,0,0,0)$ is a pure DE of the game.

Note that the analogous statements are clearly true for all other three pure strategies by exchanging the corresponding indices.   
\end{example}

We saw in Proposition~\ref{proposition:constructible} that the set of all games $X$ of a fixed size such that $\zar{ \mathcal{V}_X \setminus \mathcal{W}} = \mathcal{V}$ (which, by Theorem~\ref{theorem:bounds}, implies that $\mathcal{V} \cap \overline{\Delta}$ equals the set of DE of $X$) is the complement of an affine variety in the ambient affine space.
In the following, we determine this set explicitly by its defining non-equations for $(2\times 2)$-games. As we will see, it is the complement of a finite union of linear subspaces of $\mathbb{C}^8$ and the restriction that it poses is also sensible in game-theoretic terms: It just asks for both players $i \in \{1,2\}$ that, if $i$ decides to choose a pure strategy, then $i$'s payoff is still dependent on the choice of pure strategy of the other player. This equals the usual notion of ``generic $(2 \times 2)$-game'' within the game theory community and  is made precise in (2) of the following theorem.

As usual, for an ideal $I$ in a polynomial ring (or a set of polynomials), we will denote by $V(I)$ the set of all points that are zeros of all polynomials in $I$ and by $D(I)$ its complement.

\begin{theorem}\label{theorem:2x2}
For a $(2\times 2)$-game $X$ with payoff tables
\[
\begin{pmatrix}
a_{11} & a_{12}\\
a_{21} & a_{22} 
\end{pmatrix} \text{ and }
\begin{pmatrix}
b_{11} & b_{12}\\
b_{21} & b_{22} 
\end{pmatrix}
\]
and joint mixed strategies denoted by $p = (p_{11},p_{21},p_{12},p_{22})$, the following are equivalent:
\begin{enumerate}
\item $\mathcal{V} = \zar{\mathcal{V} \setminus \mathcal{W}}$ where $\mathcal{V}$ is the Spohn variety of $X$ and $\mathcal{W}$ is the union of planes defined by $ (p_{ij}+p_{kl}) = 0$ for $(i,j) \neq (k,l)$ with $i = k$ or $j = l$.
\item $a_{ii} \neq a_{ij}$, $a_{ii} \neq a_{ji}$, $b_{ii} \neq b_{ij}$, and $b_{ii} \neq b_{ji}$ for all $i,j \in \{1,2\}$ with $i \neq j$.
\end{enumerate}
In particular, if (2) holds then the set of DE of $X$ equals the set of non-negative real points of $\mathcal{V}$ and every Nash equilibrium is a dependency equilibrium of $X$.
\end{theorem}

\begin{proof}
The last sentence of the theorem follows from the equivalence above together with Theorem~\ref{theorem:bounds} and Corollary~\ref{corollary:NEisDEgeneric}. So, it is left to prove the equivalence of (1) and (2).

As noted in Remark~\ref{remark:constructible}, (1) is equivalent to $\mathcal{W}$ not containing any irreducible component of $\mathcal{V}$. This boils down to the following: No irreducible component $p_{ij} + p_{kl} = 0$ of $\mathcal{W}$ contains an irreducible component of $\mathcal{V}$. This, in turn, is equivalent to no irreducible component of $\mathcal{W}$ containing infinitely many points of $\mathcal{V}$ because $\mathcal{V}$ is defined by two equations in a three-dimensional space and therefore each of its irreducible components has dimension at least $1$. 

We will now exactly determine those cases in which an irreducible component of $\mathcal{W}$ contains infinitely many points from $\mathcal{V}$. We do it for the irreducible component $\mathcal{W}_1$ defined by $p_{11} + p_{21} = 0$ via a case distinction. The characterization for the other components follows by a suitable permutation of the indices.
After using the relation $p_{21} = -p_{11}$, we get the following two equations for the points of $\mathcal{V}$ lying in $\mathcal{W}_1$:
\begin{align}
\label{equation1}(p_{22}-p_{11})(a_{11}p_{11} + a_{12}p_{12}) - (p_{11}+p_{12})(a_{22}p_{22} - a_{21}p_{11}) &= 0 \\
\label{equation2}(p_{12}+p_{22})p_{11}(b_{11} - b_{21}) &= 0
\end{align}

\noindent
\textbf{Case 1.} $b_{11} - b_{21} \neq 0$. 
Equation~(\ref{equation2}) leads us to the two subcases 1.1 and 1.2.

\noindent
\textbf{Case 1.1.} $p_{11} = 0$. It follows by our initial relation that $p_{21} = 0$ and hence equation~(\ref{equation1}) reduces to $p_{22}p_{12}(a_{12}-a_{22}) = 0$. If now $a_{12}-a_{22} \neq 0$ then $p_{22} = 0$ or $p_{12} = 0$, so there are only the two projective points $(0,0,1,0)$ and $(0,0,0,1)$ left. If, on the other hand, $a_{12} - a_{22} = 0$ then arbitrary values of $p_{22}$ and $p_{12}$ give points that satisfy the equations and hence, in this case, that is, when $X$ lies in 
\[
D(b_{11} - b_{21}) \cap V(a_{12}- a_{22}),
\]
we have infinitely many points of $\mathcal{V}$ lying on $\mathcal{W}_1$.

\noindent
\textbf{Case 1.2.} $p_{12} = - p_{22}$. Here, equation~(\ref{equation1}) reduces to
\[
(p_{22}-p_{11})((a_{11}-a_{21})p_{11} + (a_{22}-a_{12})p_{22}) = 0.
\]
There are two ways in which this can be satisfied: Either $p_{22} - p_{11} = 0$, and hence there is the single projective point $(x,-x,-x,x)$ satisfying this, or $(a_{11}-a_{21})p_{11} + (a_{22}-a_{12})p_{22} = 0$. If at least one of the two terms $a_{11}-a_{21}$ and $a_{22}-a_{12}$ is non-zero then this gives only one projective point. If, however, both of these equal $0$ then all projective points of the form $(p_{11},-p_{11},-p_{22},p_{22})$ satisfy the equations. This is in case $X$ lies in 
\[
D(b_{11} - b_{21}) \cap V(a_{11}- a_{21},a_{22}-a_{11}).
\]

\noindent
\textbf{Case 2.} $b_{11} -b_{21} = 0$. Then, we only have to consider equation~(\ref{equation1}). Looking at the affine open subset of $\mathbb{P}^3$ with $p_{11} = 1$, this reduces to
\[
(a_{12}-a_{22})p_{12}p_{22} + (a_{21}-a_{12})p_{12} + (a_{11} - a_{22})p_{22} +(a_{21}-a_{11}) = 0.
\]
This equation has infinitely many solutions unless the polynomial in the variables $p_{12}$ and $p_{22}$ on the left is a non-zero constant or, equivalently, $a_{12} = a_{22}$, $a_{21} = a_{12}$, $a_{11} = a_{22}$ and $a_{21} \neq a_{11}$ which is impossible. So, for all games in
\[
V(b_{11}-b_{21})
\]
the Spohn variety has an irreducible component inside $\mathcal{W}_1$. To sum it up, the set of games for which this happens equals
\[
V(b_{11}-b_{21}) \cup V(a_{12}-a_{22}) \cup V(a_{11}-a_{21},a_{22}-a_{11}).
\]
Analogous considerations for the other components of $\mathcal{W}$ lead to the three sets
\[
V(b_{12}-b_{22}) \cup V(a_{11}-a_{21}) \cup V(a_{12}-a_{22},a_{11}-a_{22}),
\]
\[
V(a_{11}-a_{12}) \cup V(b_{22}-b_{21}) \cup V(b_{11}-b_{12},b_{22}-b_{11})
\]
and
\[
V(a_{21}-a_{22}) \cup V(b_{11}-b_{12}) \cup V(b_{21}-b_{22},b_{11}-b_{22}).
\]
The union of these four sets gives the set of all games such $a_{ii} = a_{ij}$ or $a_{ii} = a_{ji}$ or $b_{ii} = b_{ij}$ or $b_{ii} = b_{ji}$ for some $i,j \in \{1,2\}$ with $i \neq j$. This proves the statement.
\end{proof}

\begin{remark}
Note that the four pure strategies for $(2\times 2)$-games are always DE of a game satisfying (2) from Theorem~\ref{theorem:2x2} because the pure strategies always lie on the Spohn variety.
\end{remark}

\section*{Acknowledgements} 
We would like to thank Bernd Sturmfels for his helpful advice throughout this project. We are also grateful to Elke Neuhaus for the engaging discussions on the topic, and to the anonymous referee for their detailed and constructive feedback.

\bibliographystyle{amsplainurl}
\bibliography{bibliography}
 
%

\end{document}